\documentclass[reqno]{amsart}

\usepackage{amsmath}
\usepackage{amsfonts}
\usepackage{amssymb,enumerate}
\usepackage{amsthm}
\usepackage[all]{xy}
\usepackage{rotating,color}
\usepackage{hyperref}
\theoremstyle{plain}
\newtheorem{lem}{Lemma}[section]

\newtheorem{prop}[lem]{Proposition}
\newtheorem{thm}[lem]{Theorem}

\newtheorem{intthm}{Theorem}

\theoremstyle{definition}
\newtheorem{defn}[lem]{Definition}
\newtheorem{ex}[lem]{Example}

\newtheorem{disc}[lem]{Remark}
\newtheorem{construction}[lem]{Construction}

\newtheorem{fact}[lem]{Fact}

\newtheorem{claim}{Claim}

\newtheorem*{convention*}{Convention}

\newtheorem{Proof}[lem]{Proof}

\newcommand{\catd}{\mathcal{D}}

\newcommand{\id}{\operatorname{id}}

\newcommand{\HH}{\operatorname{H}}

\newcommand{\im}{\operatorname{Im}}
\newcommand{\shift}{\mathsf{\Sigma}}

\newcommand{\Ker}{\operatorname{Ker}}

\newcommand{\ol}{\overline}
\newcommand{\wti}{\widetilde}

\newcommand{\bbz}{\mathbb{Z}}

\newcommand{\xra}{\xrightarrow}

\newcommand{\into}{\hookrightarrow}

\renewcommand{\geq}{\geqslant}
\renewcommand{\leq}{\leqslant}
\renewcommand{\ker}{\Ker}
\renewcommand{\hom}{\Hom}

\newcommand{\Ext}[4][R]{\operatorname{Ext}_{#1}^{#2}(#3,#4)}	
\newcommand{\Rhom}[3][R]{\mathbf{R}\!\operatorname{Hom}_{#1}(#2,#3)}	

\newcommand{\Otimes}[3][R]{#2\otimes_{#1}#3}
\newcommand{\Hom}[3][R]{\operatorname{Hom}_{#1}(#2,#3)}

\newcommand{\und}[1]{#1^{\natural}}

\newcommand{\yext}[4][R]{\operatorname{YExt}_{#1}^{#2}(#3,#4)}

\numberwithin{equation}{lem}

\begin{document}

\bibliographystyle{amsplain}

\title{Extension Groups for DG Modules}

\author{Saeed Nasseh}

\address{Saeed Nasseh, 
Department of Mathematical Sciences,
Georgia Southern University,
Statesboro, Georgia 30460, USA}

\email{snasseh@georgiasouthern.edu}
\urladdr{https://cosm.georgiasouthern.edu/math/saeed.nasseh}

\author{Sean Sather-Wagstaff}

\address{Sean Sather-Wagstaff, 
Department of Mathematical Sciences, 
Clemson University, 
O-110 Martin Hall, Box 340975, 
Clemson, S.C. 29634, 
USA}

\email{ssather@clemson.edu}
\urladdr{https://ssather.people.clemson.edu/}

\thanks{Sather-Wagstaff  was supported in part by North Dakota EPSCoR, 
National Science Foundation Grant EPS-0814442,
and NSA grant H98230-13-1-0215.}




\begin{abstract}
Let $M$ and $N$ be differential graded (DG) modules 
over a positively graded commutative DG algebra $A$.
We show that the Ext-groups $\Ext[A]iMN$ defined in terms of semi-projective resolutions are not
in general isomorphic to the Yoneda Ext-groups $\yext[A]iMN$ given in terms of
equivalence classes of extensions. On the other hand, we show that these groups are
isomorphic when the first DG module is semi-projective.
\end{abstract}

\maketitle


\section{Introduction} \label{sec0}

\begin{convention*}\label{conv110211a}
In this paper, $R$ is a commutative  
ring with identity.
\end{convention*}

Given two $R$-modules $M$ and $N$, a classical result originating with work of Baer~\cite{baer}
states that  $\Ext 1MN$, defined via projective/injective resolutions,
is isomorphic to the abelian group $\yext 1MN$ of equivalence classes of
exact sequences of the form $0\to N\to X\to M\to 0$. The purpose of this note is to discuss possible extensions of this
result to the
abelian category of differential graded (DG) modules over a positively graded commutative DG algebra $A$.
See Section~\ref{sec110211a} for background information on this category.

Specifically, we show that Baer's isomorphism fails in general in this context:
Examples~\ref{ex111018a'} and~\ref{ex111018a} exhibit DG $A$-modules
$M$, $N$ with $\Ext[A] 1MN\ncong\yext[A] 1MN$.
(See~\ref{disc110616a} and~\ref{fact110218d} below for definitions.)
On the other hand, the following result shows that a reasonable hypothesis on the first module
does yield such an isomorphism.

\begin{intthm}\label{lem110922a}
Let $A$ be a    DG $R$-algebra, and let $N$, $Q$ be 
DG $A$-modules such that $Q$ is semi-projective.  
Then there is an isomorphism  $\yext[A]iQN\cong\Ext[A]iQN$ of abelian groups for all $i\geq 1$.
\end{intthm}

This  is the main result of Section~\ref{sec111103a};
see Proof~\ref{proof160326a}. 
In the subsequent Section~\ref{sec160326a}, we discuss some properties of YExt with respect to truncations.

It is worth noting here that we apply results from this paper in
our answer to a question of Vasconcelos in~\cite{nasseh:gart}. 
Specifically, in that paper, we investigate DG $A$-modules $C$ with $\Ext[A]1CC=0$ using geometric techniques.
These techniques yield an isomorphism between $\yext[A]1CC$ and a certain quotient of tangent spaces;
it is then important for us to know when the vanishing of $\Ext[A]1CC$ implies the vanishing of related $\operatorname{YExt}^1$-modules;
see Proposition~\ref{prop111110a} below.

\section{DG Modules}
\label{sec110211a}

We assume that the reader is familiar with the category of $R$-complexes and
the derived category $\catd(R)$.
Standard references for these topics 
are~\cite{christensen:gd,christensen:dcmca,gelfand:moha,hartshorne:rad,verdier:cd,verdier:1}.
For clarity, we include some definitions and notation.

\begin{defn}
  \label{cx}
In this paper, complexes of $R$-modules (``$R$-complexes'' for short) are indexed homologically:
$M = \cdots \xra{\partial_{n+2}^{M}} M_{n+1} \xra{\partial_{n+1}^{M}} M_n
\xra{\partial_{n}^{M}} M_{n-1} \xra{\partial_{n-1}^{M}} \cdots$.
The degree of an element $m\in M$ is denoted $|m|$.
The \emph{infimum} and \emph{supremum} of $M$ are the infimum and supremum, respectively, of the set
$\{n\in\bbz\mid\HH_n(M)\neq 0\}$.
The \emph{tensor product} of two $R$-complexes $M,N$ is denoted $\Otimes MN$,
and the \emph{Hom complex} is denoted $\Hom MN$.
A \emph{chain map} $M\to N$ is a cycle  in $\Hom MN_0$.
\end{defn}

Next we discuss  DG algebras and DG modules, which are treated  in, e.g., 
\cite{avramov:ifr,avramov:dgha,avramov:bsolrhoffd,beck:sgidgca,felix:rht,keller:ddgc}.
We follow the notation and terminology from~\cite{avramov:dgha,beck:sgidgca}; given the slight differences in the literature,
though, we include  a summary next.

\begin{defn}
  \label{DGK}
  A \emph{positively graded commutative differential graded $R$-algebra} (\emph{DG $R$-algebra} for short)
  is an $R$-complex $A$ equipped with a
  chain map
  $\mu^A\colon A\otimes_RA\to A$ with $ab:=\mu^A(a\otimes b)$
  that is associative, unital, and graded commutative such that $A_i=0$ for $i<0$.
The map $\mu^A$ is the \emph{product} on $A$.
Given a DG $R$-algebra $A$, the \emph{underlying algebra} is the
graded commutative  $R$-algebra
$\und{A}=\oplus_{i\geq 0} A_i$.

A \emph{differential graded module over} a DG $R$-algebra $A$
(\emph{DG $A$-module} for short) is an $R$-complex $M$ with a
chain map
$\mu^M\colon \Otimes AM\to M$ such that the rule $am:=\mu^M(a\otimes m)$
is associative and unital.
The map $\mu^M$ is the \emph{scalar multiplication} on $M$.
The \emph{underlying $\und{A}$-module} associated to $M$ is the
$\und{A}$-module
$\und{M}=\oplus_{i\in\bbz} M_i$.

The DG $A$-module $\Hom[A]MN$ is the subcomplex of $\Hom MN$  of
the $A$-linear homomorphisms.
A \emph{morphism} $M\to N$ of DG $A$-modules is a cycle in $\Hom[A] MN_0$.
Projective objects in the category of DG $A$-modules are called \emph{categorically
projective}.
Quasiisomorphisms of DG $A$-modules are identified by the
symbol $\simeq$,  also used for the ``quasiisomorphic'' equivalence relation.
\end{defn}

Two important DG $R$-algebras to keep in mind are $R$ itself and, more generally, 
the Koszul complex over $R$ (on a finite sequence of elements of $R$)
with the exterior product.
A DG $R$-module is just an $R$-complex, and a morphism of DG $R$-modules is simply a chain map.

\begin{disc}\label{disc110616a''}
Let $A$ be a    DG $R$-algebra.
The category of DG $A$-modules is an  abelian category with enough projectives.
\end{disc}

\begin{defn}\label{disc110616a}
Let $A$ be a    DG $R$-algebra, and let $M$, $N$ be DG $A$-modules.
For each $i\geq 0$
we have a well-defined \emph{Yoneda Ext group} $\yext[A] iMN$,
defined in terms of a resolution of $M$ by categorically projective DG $A$-modules:
$$\cdots \to Q_1\to Q_0\to M\to 0.$$
A standard result shows that  $\yext[A] 1MN$
is isomorphic to the set of equivalence classes of 
exact sequences 
$0\to N\to X\to M\to 0$
of DG $A$-modules under the Baer sum; see, e.g., \cite[(3.4.6)]{weibel:iha} and the proof of Theorem~\ref{thm131221a}.
\end{defn}

We now turn to the derived category $\catd(A)$, and related notions.

\begin{defn}
  \label{DGK2}
Let $A$ be a DG $R$-algebra.
A DG $A$-module $Q$ is \emph{graded-projective} if
$\Hom[A]Q-$ preserves surjective morphisms, that is, if
$\und Q$ is a projective graded $\und R$-module.
The DG module $Q$ is \emph{semi-projective} if $\Hom[A]Q-$ respects surjective quasiisomorphisms,
that is, if $Q$ is graded-projective and respects quasiisomorphisms.
A \emph{semi-projective resolution} of  $M$ is a quasiisomorphism
$L\xra{\simeq}M$ of DG $A$-modules such that $L$ is semi-projective.
\end{defn}

\begin{fact}\label{fact110218d}
Let $A$ be a DG $R$-algebra.
Then every DG $A$-module has a semi-projective resolution.
\end{fact}

\begin{defn}\label{def160326a}
Let $A$ be a DG $R$-algebra.
The derived category $\catd(A)$ is formed from the category of DG $A$-modules
by formally inverting the quasiisomorphisms; see~\cite{keller:ddgc}.
Isomorphisms in $\catd(A)$  are identified by the symbol $\simeq$.

The  derived functor $\Rhom[A]MN$ is defined via a semi-projective resolution
$P\xra\simeq M$,
as $\Rhom[A]MN\simeq\Hom[A]PN$.
For each $i\in\bbz$,
set  $\Ext[A]iMN:=\HH_{-i}(\Rhom[A]MN)$.
\end{defn}

\section{DG Ext vs. Yoneda Ext}\label{sec111103a}

We begin this section with  examples of DG $A$-modules 
$M$ and $N$ such that $\Ext[A] 1MN\ncong\yext[A] 1MN$.
These present two facets of the distinctness of Ext and YExt,
as the first example has $M$ and $N$ both bounded, while the second one (from 
personal communication with Avramov) has $M$ graded-projective.

\begin{ex}\label{ex111018a'}
Let $R=k[\![X]\!]$, and consider the following exact sequence of DG $R$-modules,
i.e., exact sequence of $R$-complexes:
$$\xymatrix{
0\ar[r]
&\underline R\ar[r]
&\underline R\ar[r]
&\underline k\ar[r]
&0 \\
& 0\ar[d]
& 0\ar[d]
& 0\ar[d] \\
0\ar[r]
& R\ar[r]^-X\ar[d]_-1
& R\ar[r]\ar[d]_-1
& k\ar[r]\ar[d]_-1
&0 \\
0\ar[r]
& R\ar[r]^-X\ar[d]
& R\ar[r]\ar[d]
& k\ar[r]\ar[d]
&0 \\
&0&0&0.
}$$
This sequence does not split over $R$ (it is not even degree-wise split) so 
it gives a non-trivial class in $\yext 1{\underline k}{\underline R}$,
and we conclude that $\yext 1{\underline k}{\underline R}\neq 0$.
On the other hand, $\underline k$ is homologically trivial,
so we have $\Ext 1{\underline k}{\underline R}=0$ since
$0$ is a semi-free resolution of $\underline k$.
\end{ex}

\begin{ex}\label{ex111018a}
Let $R=k[X]/(X^2)$ and consider the following exact graded-projective DG $R$-module
$M=\cdots\xra XR\xra XR\xra X\cdots$.
Since $M$ is exact, we have $\Ext iMM=0$ for all $i$.
We claim, however, that $\yext 1MM\neq 0$. 
To see this, first note that $M$ is isomorphic to the suspension $\shift M$ and that $M$ is not contractible.
Thus, the mapping cone sequence for the identity morphism $\id_M$
is isomorphic to one of the form
$0\to M\to X\to M\to 0$
and is not split.
\end{ex}

The  definition of the isomorphism $\yext[A]iQN\to\Ext[A]iQN$
for $i=1$ in Theorem~\ref{lem110922a} is contained in the following construction. 
The subsequent lemma and theorem show that $\Psi$ is a well-defined isomorphism.

\begin{construction}\label{c131221a}
Let $A$ be a DG $R$-algebra, and let $N$, $Q$ be 
DG $A$-modules such that $Q$ is graded-projective.  
Define $\Psi\colon\yext[A]1QN\to\HH_{-1}(\Hom[A]QN)$ as follows. 
Note that if $Q$ is semi-projective, then $\Ext[A]1QN\cong\HH_{-1}(\Hom[A] QN)$,
which fits with what we have in Theorem~\ref{lem110922a}.

Let $\zeta\in\yext[A]1QN$ be represented by the  sequence
\begin{equation}\label{eq111018a}
0\to N\xra f X\xra g Q\to 0.
\end{equation}
Since $Q$ is graded-projective, this sequence
is \emph{graded-split}, that is
there are elements $h\in\hom[A]XN_0$ and $k\in\Hom[A]QX_0$ with
\begin{align*}
hf&=\id_N
&gk&=\id_Q
&hk&=0
&fh+kg&=\id_X.
\end{align*}
Thus, the sequence~\eqref{eq111018a} is isomorphic
to one of the form
\begin{equation}\label{eq111018b}
\begin{split}
\xymatrix{
&\vdots\ar[d]_-{\partial^N_{i+1}}
&\vdots\ar[d]_-{\partial^X_{i+1}}
&\vdots\ar[d]_-{\partial^Q_{i+1}}
\\
0\ar[r]
&N_i\ar[d]_{\partial^N_{i}}\ar[r]^-{\epsilon_i}
&N_i \oplus  Q_i\ar[d]_{\partial^X_{i}}\ar[r]^-{\pi_i}
&Q_i\ar[d]_{\partial^Q_{i}}\ar[r]
&0
\\
0\ar[r]
&N_{i-1}\ar[d]_-{\partial^N_{i-1}}\ar[r]^-{\epsilon_{i-1}}
&N_{i-1} \oplus Q_{i-1}\ar[d]_-{\partial^X_{i-1}}\ar[r]^-{\pi_{i-1}}
&Q_{i-1}\ar[d]_-{\partial^Q_{i-1}}\ar[r]
&0
\\
&\vdots
&\vdots
&\vdots
}
\end{split}
\end{equation}
where $\epsilon_j$ is the natural inclusion and $\pi_j$ is the natural surjection
for each $j$.
Since this diagram comes from a graded-splitting of~\eqref{eq111018a},
the scalar multiplication on the middle column of~\eqref{eq111018b} is the natural one
$a\left[\begin{smallmatrix}p\\ q\end{smallmatrix}\right]=\left[\begin{smallmatrix}ap\\ aq\end{smallmatrix}\right]$.
(We write elements of $N_i\oplus Q_i$ as column vectors.) 

The fact that~\eqref{eq111018b} commutes implies that $\partial^X_i$ has a specific form:
\begin{equation}\label{eq111018c}
\partial^X_i=
\left[\begin{smallmatrix}
\partial^N_i & \lambda_i \\
0 & \partial^Q_i
\end{smallmatrix}\right].
\end{equation}
Here, we have $\lambda_i\colon Q_i\to N_{i-1}$, that is,
$\lambda=\{\lambda_i\}\in\Hom Q{N}_{-1}$. Since the 
horizontal maps in the sequence~\eqref{eq111018b}
are morphisms of DG $A$-modules, it follows that $\lambda$ is a cycle in
$\Hom[A] Q{N}_{-1}$. Thus, $\lambda$ represents a homology class in
$\HH_{-1}(\Hom[A] QN)$, and we define $\Psi\colon\yext[A]1QN\to\HH_{-1}(\Hom[A] QN)$ by setting
$\Psi(\zeta)$ equal to $[\lambda]$ the homology class of $\lambda$ in $\HH_{-1}(\Hom[A] QN)$.
\end{construction}

\begin{lem}\label{lem131221azz}
Let $A$ be a DG $R$-algebra, and let $N$, $Q$ be 
DG $A$-modules such that $Q$ is graded-projective.  
Then the map $\Psi\colon\yext[A]1QN\to\HH_{-1}(\Hom[A] QN)$ from Construction~\ref{c131221a}
is well-defined.
\end{lem}

\begin{proof}
Let $\zeta\in\yext[A]1QN$ be represented by the
sequence~\eqref{eq111018b}, and let $\zeta$ be represented by another exact sequence
\begin{equation}\label{eq111018d}
\begin{split}
\xymatrix{
&\vdots\ar[d]_-{\partial^N_{i+1}}
&\vdots\ar[d]_-{\partial^{X'}_{i+1}}
&\vdots\ar[d]_-{\partial^Q_{i+1}}
\\
0\ar[r]
&N_i\ar[d]_{\partial^N_{i}}\ar[r]^-{\epsilon_i}
&N_i \oplus  Q_i\ar[d]_{\partial^{X'}_{i}}\ar[r]^-{\pi_i}
&Q_i\ar[d]_{\partial^Q_{i}}\ar[r]
&0
\\
0\ar[r]
&N_{i-1}\ar[d]_-{\partial^N_{i-1}}\ar[r]^-{\epsilon_{i-1}}
&N_{i-1}\oplus  Q_{i-1}\ar[d]_-{\partial^{X'}_{i-1}}\ar[r]^-{\pi_{i-1}}
&Q_{i-1}\ar[d]_-{\partial^Q_{i-1}}\ar[r]
&0
\\
&\vdots
&\vdots
&\vdots
}
\end{split}
\end{equation}
where 
\begin{equation}\label{eq111018e}
\partial^{X'}_i=
\left[\begin{smallmatrix}
\partial^N_i & \lambda_i '\\
0 & \partial^Q_i
\end{smallmatrix}\right].
\end{equation}
We need to show that $\lambda-\lambda'\in\im(\partial_0^{\Hom[A]{Q}{N}})$.
The sequences~\eqref{eq111018b}
and~\eqref{eq111018d} are equivalent in $\yext 1QN$,
so for each $i$ there is a commutative diagram
\begin{equation}\label{eq111018f}
\begin{split}
\xymatrix{
0\ar[r]
&N_i\ar[d]_{=}\ar[r]^-{\epsilon_i}
&N_i \oplus  Q_i
\ar[d]_{\left[\begin{smallmatrix}
u_i & v_i \\
w_i & x_i
\end{smallmatrix}\right]}^{\cong}\ar[r]^-{\pi_i}
&Q_i\ar[d]_{=}\ar[r]
&0
\\
0\ar[r]
&N_{i}\ar[r]^-{\epsilon_{i}}
&N_{i}\oplus  Q_{i}\ar[r]^-{\pi_{i}}
&Q_{i}\ar[r]
&0
}
\end{split}
\end{equation}
where the middle vertical arrow describes a DG $A$-module isomorphism,
and such that the following diagram commutes for all $i$
\begin{equation}\label{eq111018g}
\begin{split}
\xymatrix@C=20mm@R=10mm{
N_i \oplus  Q_i
\ar[d]_{\left[\begin{smallmatrix}
\partial^N_i & \lambda_i \\
0 & \partial^Q_i
\end{smallmatrix}\right]}
\ar[r]_\cong^{\left[\begin{smallmatrix}
u_{i} & v_{i} \\
w_{i} & x_{i}
\end{smallmatrix}\right]}
&N_{i} \oplus  Q_{i}
\ar[d]^{\left[\begin{smallmatrix}
\partial^N_i & \lambda'_i \\
0 & \partial^Q_i
\end{smallmatrix}\right]}\\
N_{i-1} \oplus  Q_{i-1}
\ar[r]_\cong^{\left[\begin{smallmatrix}
u_{i-1} & v_{i-1} \\
w_{i-1} & x_{i-1}
\end{smallmatrix}\right]}
&N_{i-1} \oplus  Q_{i-1}.
}
\end{split}
\end{equation}
The fact that diagram~\eqref{eq111018f} commutes implies that
$u_i=\id_{N_i}$, 
$x_i=\id_{Q_i}$, and
$w_i=0$.
Also, the fact that the middle vertical arrow in diagram~\eqref{eq111018f} describes a
DG $A$-module morphism implies that the sequence
$v_i\colon Q_i\to N_i$ respects scalar multiplication, i.e., we have
$v\in\Hom[A]QN_0$.
The fact that diagram~\eqref{eq111018g} commutes implies that
$\lambda_i-\lambda_i'=\partial^N_iv_i-v_{i-1}\partial^Q_i$.
We conclude that
$\lambda-\lambda'=\partial_0^{\Hom[A]{Q}{N}}(v)\in\im(\partial_0^{\Hom[A]{Q}{N}})$,
so $\Psi$ is well-defined.
\end{proof}

The next result contains the case $i=1$ of Theorem~\ref{lem110922a} from the introduction, because
if $Q$ is semi-projective, then $\Ext[A]1QN\cong\HH_{-1}(\Hom[A] QN)$.

\begin{thm}\label{thm131221a}
Let $A$ be a DG $R$-algebra, and let $N$, $Q$ be 
DG $A$-modules such that $Q$ is graded-projective.  
Then the map $\Psi\colon\yext[A]1QN\to\HH_{-1}(\Hom[A] QN)$ from Construction~\ref{c131221a}
is a group isomorphism.
\end{thm}

\begin{proof}
We break the proof into three claims.
\begin{claim}\label{step131221a}
$\Psi$ is additive.
Let $\zeta,\zeta'\in \yext[A]1QN$ be represented by  exact sequences
$0\to N\xra{\epsilon} X\xra{\pi} Q\to 0$ and
$0\to N\xra{\epsilon'} X'\xra{\pi'} Q\to 0$
respectively, where $X_i=N_i \oplus Q_i= X_i'$
and the differentials $\partial^X$ and $\partial^{X'}$ are described as in~\eqref{eq111018c}
and~\eqref{eq111018e}, respectively.
We need to show that the Baer sum $\zeta+\zeta'$ is represented 
by an exact sequence
$0\to N\xra{\wti\epsilon} \wti X\xra{\wti\pi} Q\to 0$,
where $\wti X_i=N_i\oplus Q_i$
and 
$\partial^{\wti X}_i=
\left[\begin{smallmatrix}
\partial^N_i & \lambda_i+\lambda'_i \\
0 & \partial^Q_i
\end{smallmatrix}\right]$, with scalar multiplication 
$a\left[\begin{smallmatrix}p\\ q\end{smallmatrix}\right]=\left[\begin{smallmatrix}ap\\ aq\end{smallmatrix}\right]$.
Note that it is straightforward to show that the sequence $\wti X$ defined in this way
is a DG $A$-module, and the natural maps 
$N\xra{\wti\epsilon} \wti X\xra{\wti\pi} Q$ are $A$-linear, using the analogous properties
for $X$ and $X'$.

We construct the Baer sum in two steps.
The first step is to construct the pull-back diagram
$$\xymatrix{
X''\ar[r]\ar[d]\ar@{}[rd]|<<{\ulcorner}
&X'\ar[d]^{\pi'} \\
X\ar[r]^{\pi}&Q.
}$$
The DG module $X''$ is a submodule of the direct sum 
$X \oplus X'$, so each 
$X''_i$ is the submodule of 
$$(X \oplus X')_i
=X_i \oplus X'_i
\cong N_i \oplus Q_i \oplus N_i \oplus Q_i
$$
consisting of all vectors $\left[\begin{smallmatrix}x\\ x'\end{smallmatrix}\right]$
such that $\pi_i'(x')=\pi_i(x)$, that is, all vectors
of the form $[\begin{matrix}p& q& p' & q'\end{matrix}]^T$
such that $q=q'$. 
In other words, we have
\begin{equation}
N_i\oplus Q_i \oplus N_i\xra{\cong}X''_i
\label{eq111228a}
\end{equation}
where the isomorphism is given by
$[\begin{matrix}p& q&  p'\end{matrix}]^T
\mapsto
[\begin{matrix}p& q&  p'& q\end{matrix}]^T$.
The differential on $X\oplus X'$
is the natural diagonal map. So, under the isomorphism~\eqref{eq111228a}, 
the differential on $X''$ has the form
$$X''_i
\cong N_i\oplus Q_i \oplus N_i
\xra{\partial^{X''}_i
=\left[\begin{smallmatrix}\partial^N_i & \lambda_i &  0\\
0 & \partial^Q_i &  0 \\
0 & \lambda'_i & \partial^N_i  \end{smallmatrix}\right]}
N_{i-1} \oplus Q_{i-1} \oplus N_{i-1}
\cong X''_{i-1}.$$

The next step in the construction of $\zeta+\zeta'$ is to build $\wti X$, which is the cokernel of the morphism
$\gamma\colon N\to X''$ given by $p\mapsto\left[\begin{smallmatrix}-p\\0\\ p\end{smallmatrix}\right]$.
That is, since $\gamma$ is injective,
the complex $\wti X$ is determined by the  exact sequence
$0\to N\xra{\gamma}X''\xra{\tau}\wti X\to 0$.
It is straightforward to show that this sequence has the following form
$$\xymatrix@C=7mm{
0\ar[r]
&N_i
\ar[rr]^-{\protect{\left[\begin{smallmatrix}-1\\ 0\\ 1\end{smallmatrix}\right]}}
\ar[dd]_{\partial^N_i}
&&N_i\oplus Q_i \oplus N_i
\ar[rr]^-{\protect{\left[\begin{smallmatrix}1& 0& 1\\ 0 & 1 & 0\end{smallmatrix}\right]}}
\ar[dd]^{\protect{\left[\begin{smallmatrix}\partial^N_i & \lambda_i &  0\\
0 & \partial^Q_i &  0 \\
0 & \lambda'_i & \partial^N_i \end{smallmatrix}\right]}}
&&N_i \oplus Q_i\ar[r]
\ar[dd]^{\protect{\left[\begin{smallmatrix}\partial^N_i & \lambda_i+\lambda'_i \\
0 & \partial^Q_i\end{smallmatrix}\right]}}
&0\\ \\
0\ar[r]
&N_{i-1}
\ar[rr]^-{\protect{\left[\begin{smallmatrix}-1\\ 0\\ 1\end{smallmatrix}\right]}}
&&N_{i-1} \oplus Q_{i-1} \oplus N_{i-1}
\ar[rr]^-{\protect{\left[\begin{smallmatrix}1& 0& 1\\ 0 & 1 & 0\end{smallmatrix}\right]}}
&&N_{i-1} \oplus Q_{i-1}\ar[r]
&0.}$$
By inspecting the right-most column of this diagram, we see that $\wti X$ has the desired
form.  Furthermore, checking the module structures at each step of the construction,
we see that
the scalar multiplication on $\wti X$ is the natural one
$a\left[\begin{smallmatrix}p\\ q\end{smallmatrix}\right]=\left[\begin{smallmatrix}ap\\ aq\end{smallmatrix}\right]$.
This concludes the proof of Claim~\ref{step131221a}.
\end{claim}

\begin{claim}\label{step131221b}
$\Psi$ is injective.
Suppose that $\zeta\in\ker(\Psi)$ is represented by the 
displays~\eqref{eq111018a}--\eqref{eq111018c}.
The condition $\Psi(\zeta)=0$ says that $\lambda\in\im(\partial_0^{\Hom[A]QN})$,
so there is an element $s\in\Hom[A]QN_0$ such that $\lambda=\partial_0^{\Hom[A]QN}(s)$.
Thus, for each $i$ we have
$\lambda_i=\partial^N_is_i-s_{i-1}\partial^Q_i$.
From this, it is straightforward to show that the following diagram commutes:
$$\xymatrix@C=20mm@R=10mm{
N_i\oplus Q_i
\ar[d]_{\left[\begin{smallmatrix}
\partial^N_i & \lambda_i \\
0 & \partial^Q_i
\end{smallmatrix}\right]}
\ar[r]_\cong^{\left[\begin{smallmatrix}
1 & s_{i} \\
0 & 1
\end{smallmatrix}\right]}
&N_{i}\oplus  Q_{i}
\ar[d]^{\left[\begin{smallmatrix}
\partial^N_i & 0 \\
0 & \partial^Q_i
\end{smallmatrix}\right]}\\
N_{i-1}\oplus  Q_{i-1}
\ar[r]_\cong^{\left[\begin{smallmatrix}
1 & s_{i-1} \\
0 & 1
\end{smallmatrix}\right]}
&N_{i-1} \oplus Q_{i-1}.
}$$
From the fact that $s$ is $A$-linear, it follows that the maps
$\left[\begin{smallmatrix}
1 & s_{i} \\
0 & 1
\end{smallmatrix}\right]$ describe an $A$-linear isomorphism
$X\xra\cong N\oplus  Q$
making the following diagram commute:
$$\xymatrix{
0\ar[r]
&N \ar[d]_{=}\ar[r]^{\epsilon}
&X
\ar[d]_-\cong
\ar[r]^{\pi}
&Q\ar[d]_{=}\ar[r]
&0
\\
0\ar[r]
&N\ar[r]^-{\epsilon}
&N\oplus  Q\ar[r]^-{\pi}
&Q\ar[r]
&0.
}$$
In other words, the sequence~\eqref{eq111018a} splits, so we have $\zeta=0$,
and $\Psi$ is injective.
This concludes the proof of Claim~\ref{step131221b}.
\end{claim}

\begin{claim}\label{step131221c}
$\Psi$ is surjective.
For this, let $\xi\in\HH_{-1}(\Hom[A] QN)$ be represented by
$\lambda\in\Ker(\partial_{-1}^{\Hom[A]QN})$.
Using the fact that $\lambda$ is $A$-linear such that $\partial_{-1}^{\Hom[A]QN}(\lambda)=0$,
one checks directly that the displays~\eqref{eq111018b}--\eqref{eq111018c}
describe an exact sequence of DG $A$-module homomorphisms 
of the form~\eqref{eq111018a}
whose image under $\Psi$ is $\xi$.
This concludes the proof of Claim~\ref{step131221c} and the proof of the theorem.\qedhere
\end{claim}
\end{proof}

\begin{disc}\label{disc160314b}
After the results of this paper were announced,
Avramov, et al.~\cite{avramov:dgha} established the following generalization of Theorem~\ref{thm131221a}.
\begin{prop}
\label{prop160314a}
Let $A$ be a DG $R$-algebra, and let $M$ and $N$ be DG $A$-modules.
There is a monomorphism of abelian groups
$$\kappa\colon\HH_{0}(\Hom[A]{\shift^{-1}M}N)\to\yext[U]1{M}N$$
with image equal to the set of equivalence classes of graded-split exact sequences of the form
$0\to N\to X\to  M\to 0$.
\end{prop}
\noindent To see how this generalizes Theorem~\ref{thm131221a}, first note that 
if $M$ is graded-projective, then the map $\kappa$ is bijective, as in this case every element of 
$\yext[U]1{M}N$ is graded-split; thus, we have
$\HH_{-1}(\Hom[A]{M}N)\cong\HH_{0}(\Hom[A]{\shift^{-1}M}N)\cong\yext[U]1{M}N$.
\end{disc}

\begin{Proof}[Proof of Theorem~\ref{lem110922a}] \label{proof160326a}
Using Theorem~\ref{thm131221a}, we need only justify the isomorphism 
$\yext[A]iQN\cong\Ext[A]iQN$ for $i\geq 2$.
Let 
$$L_{\bullet}^+
=\cdots\xra{\partial^L_2} L_1
\xra{\partial^L_1}L_0
\xra\pi Q\to 0$$
be a  resolution of $Q$ by categorically projective DG $A$-modules.
Since each $L_j$ is categorically projective, we have $\yext[A]i{L_j}-=0$ for all $i\geq 1$
and $L_j\simeq 0$ for each $j$,
so we have $\Ext[A]i{L_j}-=0$ for all $i$.
Set $Q_i:=\im\partial^L_i$ for each $i\geq 1$.
Each $L_i$ is graded-projective, 
so the fact that $Q$ is graded-projective implies that each $Q_i$ is graded-projective.

Now,  a straightforward dimension-shifting argument
explains the first and third isomorphisms in the following display for $i\geq 2$:
$$\yext[A]i QN\cong\yext[A]1{Q_{i-1}}N\cong\Ext[A]1{Q_{i-1}}N\cong\Ext[A]i QN.$$
The second isomorphism is from Theorem~\ref{thm131221a}
since each $Q_i$ is graded-projective.
\qed
\end{Proof}

The next  example shows that one can have $\yext[A]0QN\not\cong\Ext[A]0QN$,
even when $Q$ is semi-free.

\begin{ex}\label{ex111103a}
Continue with the assumptions and notation of Example~\ref{ex111018a'},
and set $Q=N=\underline R$. It is straightforward to show that the morphisms
$\underline R\to \underline R$ are precisely given by multiplication by
fixed elements of $R$, so we have the first step in the next display:
$$\yext[A]0{\underline R}{\underline R}\cong R\neq 0=
\Ext[A]0{\underline R}{\underline R}.$$
The third step follows from the condition $\underline R\simeq 0$.
\end{ex}

\begin{disc}
It is perhaps worth noting that our proofs can also be used to give the isomorphisms from Theorem~\ref{lem110922a}
when $Q$ is not necessarily semi-projective, but $N$ is ``semi-injective''.
\end{disc}

\section{$\operatorname{YExt}^1$ and Truncations}\label{sec160326a}

For our work in~\cite{nasseh:gart},
we need to know how YExt respects the following notion.

\begin{defn}\label{truncations}
Let $A$ be a  DG $R$-algebra, and let $M$ be a DG $A$-module.
Given an integer $n$, 
the \emph{$n$th soft left truncation of $M$} is the complex
$$
\tau(M)_{(\leq n)}:=\cdots\to 0\to M_n/\im(\partial^M_{n+1})\to M_{n-1}\to M_{n-2}\to \cdots
$$
with differential induced by $\partial^M$.
In other words, $\tau(M)_{(\leq n)}$ is the quotient DG $A$-module
$M/M'$ where $M'$ is the following DG submodule of $M$:
$$M'=\cdots\to M_{n+2}\to M_{n+1}\to\im(\partial^M_{n+1})\to 0.$$
Note that we have $M'\simeq 0$ if and only if $n\geq \sup(M)$,
so the natural morphism $\rho\colon M\to \tau(M)_{(\leq n)}$ of DG $A$-modules
yields an isomorphism in $\catd(A)$ if and only if $n\geq \sup(M)$.
\end{defn}

\begin{prop}\label{lem111103a}
Let $A$ be a DG $R$-algebra, and let $M$ and $N$ be DG $A$-modules.
Assume that $n$ is an integer such that $N_i=0$ for all $i>n$. Then the natural map
$\yext[A]1{\tau(M)_{(\leq n)}}N\to \yext[A]1{M}N$ induced by the
morphism $\rho\colon M\to\tau(M)_{(\leq n)}$ from Definition~\ref{truncations} is a monomorphism.
\end{prop}

\begin{proof}
Let $\Upsilon$ denote the map $\yext[A]1{\tau(M)_{(\leq n)}}N\to \yext[A]1{M}N$
induced by  $\rho$.
Let $\alpha\in\ker(\Upsilon)\subseteq\yext[A]1{\tau(M)_{(\leq n)}}{N}$ be represented by the exact sequence 
\begin{equation}\label{eq111108a}
0\to N\xra{f} X\xra{g} \tau(M)_{(\leq n)}\to 0.
\end{equation}
Note that, since $N_i=0=(\tau(M)_{(\leq n)})_i$ for all $i>n$, we have $X_i=0$ for all $i>n$.
Our assumptions imply that $0=\Upsilon([\alpha])=[\beta]$ where $\beta$ comes from the following pull-back diagram: 
\begin{equation}\label{eq111108c}
\begin{split}
\xymatrix{
&&  0\ar[d] & 0\ar[d]& 0\ar[d]&\\
&  0\ar[r] & 0\ar[d]\ar[r]&  K\ar[r]^{=}\ar[d]_{\wti h}&K\ar[r]\ar[d]_{h}&0\\
\beta:&0\ar[r]  & N\ar[d]_{=}\ar[r]^-{\wti f}&\widetilde{X} \ar@{}[rd]|<<{\ulcorner}\ar[r]^-{\wti{g}}\ar[d]_{\wti{\rho}}
& M\ar[r]\ar[d]_{\rho}&0\\
\alpha:&0\ar[r]  & N\ar[d]\ar[r]^-f& X\ar[d]\ar[r]^-g& \tau(M)_{(\leq n)}\ar[d]\ar[r]&0\\
& & 0 & 0& 0.
}
\end{split}
\end{equation}
The middle row $\beta$ of this diagram is split exact
since $[\beta]=0$, so there is a 
morphism $F\colon \wti X\to N$ of DG $A$-modules
such that $F\circ \wti f=\id_N$.
Note that $K$ has the form
\begin{equation}\label{eq111108d}
K=\cdots\xra{\partial^M_{n+2}}M_{n+1}\xra{\partial^M_{n+1}}\im(\partial^M_{n+1})\to 0
\end{equation}
because of the right-most column of the diagram.

We claim that $F\circ\wti h=0$.
It suffices to check this degree-wise.
When $i>n$, we have $N_i=0$, so $F_i=0$, and  $F_i\circ\wti h_i=0$.
When $i<n$, the display~\eqref{eq111108d}
shows that $K_i=0$, so $\wti h_i=0$, and  $F_i\circ\wti h_i=0$.
For $i=n$, we first note that the display~\eqref{eq111108d}
shows that $\partial^K_{n+1}$ is surjective.
In the following diagram, the faces with solid arrows commute because
$\wti h$ and $F$ are morphisms:
$$\xymatrix{
0\ar@{..>}[rd]\ar@{..>}[dd]
&&K_{n+1}\ar[rd]^-{\wti h_{n+1}}\ar@{->>}[dd]_<<<<<{\partial^{K}_{n+1}}\ar@{..>}[ll]
\\
&0\ar[dd]
&&\wti X_{n+1}\ar[dd]^-{\partial^{\wti X}_{n+1}}\ar[ll]^>>>>>>>{F_{n+1}}
\\
0\ar@{..>}[rd]
&&K_{n}\ar[rd]^-{\wti h_{n}}\ar@{..>}[ll]
\\
&N_n
&&\wti X_{n}\ar[ll]^-{F_n}
}$$
Since $\partial^K_{n+1}$ is surjective, a simple diagram chase shows that
$F_n\circ\wti h_n=0$.
This establishes the claim.

To conclude the proof, note that the previous claim shows that the map
$K\to 0$ is a left-splitting of the top row of diagram~\eqref{eq111108c}
that is compatible with the left-splitting $F$ of the middle row. 
It is now straightforward to show that $F$ induces a 
morphism $\ol F\colon X\to N$ of DG $A$-modules that left-splits
the bottom row of diagram~\eqref{eq111108c}.
Since this row represents $\alpha\in\yext[A]1{\tau(M)_{(\leq n)}}{N}$,
we conclude that $[\alpha]=0$, so $\Upsilon$ is a monomorphism.
\end{proof}

The next example shows that the monomorphism from Proposition~\ref{lem111103a}
may not be an isomorphism.

\begin{ex}\label{ex111103b}
Continue with the assumptions and notation of Example~\ref{ex111018a'}.
The following diagram describes a non-zero element of
$\yext[R]1MN$:
$$\xymatrix{
0\ar[r]
&N\ar[r]
&\underline R\ar[r]
&M\ar[r]
&0 \\
& 0\ar[d]
& 0\ar[d]
& 0\ar[d] \\
0\ar[r]
& 0\ar[r]\ar[d]
& R\ar[r]^1\ar[d]_-1
& R\ar[r]\ar[d]_-\pi
&0 \\
0\ar[r]
& R\ar[r]^-X\ar[d]
& R\ar[r]^\pi\ar[d]
& k\ar[r]\ar[d]
&0 \\
&0&0&0.
}$$
It is straightforward to show that $\tau(M)_{(\leq 0)}=0$, so we have
$$0=\yext[A]1{\tau(M)_{(\leq 0)}}N\into \yext[A]1{M}N\neq 0$$
thus this map is not an isomorphism.
\end{ex}

\begin{prop}\label{prop111110a}
Let $A$ be a DG $R$-algebra, and let $C$ be a semi-projective 
DG $A$-module such that $\Ext 1CC=0$. For  $n\geq\sup(C)$, one has
$$\yext[A]1CC=0=\yext[A]1{\tau(C)_{(\leq n)}}{\tau(C)_{(\leq n)}}.$$
\end{prop}

\begin{proof}
From Theorem~\ref{thm131221a}, we have
$\yext[A]1CC\cong\Ext[A]1CC=0$.
For the remainder of the proof, assume without loss of generality that
$\sup(C)<\infty$. 
Another application of Theorem~\ref{thm131221a} explains the first step
in the next display:
\begin{align*}
\yext[A]1C{\tau(C)_{(\leq n)}}
\cong\Ext[A]1C{\tau(C)_{(\leq n)}}
\cong
\Ext[A]1CC=0.
\end{align*}
The second step comes from the assumption $n\geq\sup(C)$
which guarantees that the natural morphism
$C\to \tau(C)_{(\leq n)}$ represents an isomorphism in $\catd(A)$.
Proposition~\ref{lem111103a} implies that
$\yext[A]1{\tau(C)_{(\leq n)}}{\tau(C)_{(\leq n)}}$ is isomorphic to a 
subgroup of $\yext[A]1C{\tau(C)_{(\leq n)}}=0$, so  
$\yext[A]1{\tau(C)_{(\leq n)}}{\tau(C)_{(\leq n)}}=0$, as desired.
\end{proof}

\section*{Acknowledgments}
We are grateful to Luchezar Avramov for helpful discussions about this work.


\begin{thebibliography}{10}

\bibitem{avramov:ifr}
L.~L. Avramov, \emph{Infinite free resolutions}, Six lectures on commutative
  algebra (Bellaterra, 1996), Progr. Math., vol. 166, Birkh\"auser, Basel,
  1998, pp.~1--118. \MR{99m:13022}

\bibitem{avramov:dgha}
L.~L. Avramov, H.-B.\ Foxby, and S.\ Halperin, \emph{Differential graded
  homological algebra}, in preparation.

\bibitem{avramov:bsolrhoffd}
L.~L. Avramov, H.-B.Foxby, and J.\ Lescot, \emph{Bass series of local ring
  homomorphisms of finite flat dimension}, Trans. Amer. Math. Soc. \textbf{335}
  (1993), no.~2, 497--523. \MR{93d:13026}

\bibitem{baer}
R.~Baer, \emph{Erweiterung von {G}ruppen und ihren {I}somorphismen}, Math. Z.
  \textbf{38} (1934), no.~1, 375--416. \MR{1545456}

\bibitem{beck:sgidgca}
K.~Beck and S.~Sather-Wagstaff, \emph{A somewhat gentle introduction to
  differential graded commutative algebra}, Connections Between Algebra,
  Combinatorics, and Geometry, Proceedings in Mathematics and Statistics,
  vol.~76, Springer, New York, Heidelberg, Dordrecht, London, 2014, pp.~3--99.

\bibitem{christensen:gd}
L.~W. Christensen, \emph{Gorenstein dimensions}, Lecture Notes in Mathematics,
  vol. 1747, Springer-Verlag, Berlin, 2000. \MR{2002e:13032}

\bibitem{christensen:dcmca}
L.~W. Christensen, H.-B. Foxby, and H.~Holm, \emph{Derived category methods in
  commutative algebra}, in preparation.

\bibitem{felix:rht}
Y.\ F{\'e}lix, S.\ Halperin, and J.-C.\ Thomas, \emph{Rational homotopy
  theory}, Graduate Texts in Mathematics, vol. 205, Springer-Verlag, New York,
  2001. \MR{1802847}

\bibitem{gelfand:moha}
S.~I. Gelfand and Y.~I. Manin, \emph{Methods of homological algebra},
  Springer-Verlag, Berlin, 1996. \MR{2003m:18001}

\bibitem{hartshorne:rad}
R.~Hartshorne, \emph{Residues and duality}, Lecture Notes in Mathematics, No.
  20, Springer-Verlag, Berlin, 1966. \MR{36 \#5145}

\bibitem{keller:ddgc}
B.~Keller, \emph{Deriving {DG} categories}, Ann. Sci. \'Ecole Norm. Sup. (4)
  \textbf{27} (1994), no.~1, 63--102. \MR{1258406 (95e:18010)}

\bibitem{nasseh:gart}
S.~Nasseh and S.~Sather-Wagstaff, \emph{Geometric aspects of representation
  theory for {DG} algebras: answering a question of {V}asconcelos}, preprint
  (2012), \texttt{arXiv:1201.0037}.

\bibitem{verdier:cd}
J.-L.\ Verdier, \emph{Cat\'{e}gories d\'{e}riv\'{e}es}, SGA 4$\frac{1}{2}$,
  Springer-Verlag, Berlin, 1977, Lecture Notes in Mathematics, Vol. 569,
  pp.~262--311. \MR{57 \#3132}

\bibitem{verdier:1}
\bysame, \emph{Des cat\'egories d\'eriv\'ees des cat\'egories ab\'eliennes},
  Ast\'erisque (1996), no.~239, xii+253 pp. (1997), With a preface by Luc
  Illusie, Edited and with a note by Georges Maltsiniotis. \MR{98c:18007}

\bibitem{weibel:iha}
C.~A. Weibel, \emph{An introduction to homological algebra}, Cambridge Studies
  in Advanced Mathematics, vol.~38, Cambridge University Press, Cambridge,
  1994. \MR{1269324 (95f:18001)}

\end{thebibliography}
\providecommand{\bysame}{\leavevmode\hbox to3em{\hrulefill}\thinspace}
\providecommand{\MR}{\relax\ifhmode\unskip\space\fi MR }
\providecommand{\MRhref}[2]{%
  \href{http://www.ams.org/mathscinet-getitem?mr=#1}{#2}
}
\providecommand{\href}[2]{#2}

\end{document}